\documentclass[a4paper, twoside,12pt]{article}
\usepackage{fancyhdr}
\usepackage{fnpos}
 \usepackage[english]{babel}        
\usepackage[T1]{fontenc}          
\usepackage{graphicx}             
\usepackage{makeidx}
\usepackage{fancybox}
\usepackage{framed}
\usepackage{fancyhdr}
 \usepackage{pstricks,pst-plot,pstricks-add}
\usepackage[margin=1in]{geometry}
\usepackage{graphicx}
\usepackage{titlesec}
\usepackage{amsmath}
\usepackage{amsfonts}   
\usepackage{amssymb}    
\usepackage{amsthm}
\usepackage{dsfont}
\usepackage{mathtools}
\usepackage{verbatim}   
\usepackage{color}
\usepackage{listings}
\usepackage{graphicx}
\usepackage{amsmath}
\usepackage{amsmath,amssymb,color,inputenc,euscript,graphicx,psfrag}

\newtheorem{thm}{Theorem}[section]

\newtheorem{lem}[thm]{Lemma}
\newtheorem{prop}[thm]{Proposition}

\numberwithin{equation}{section}

\renewcommand{\thefootnote}

\renewcommand\Re{\operatorname{Re}}

 {  }{  }

\author { B\'echir Amri and Abdessalem Guesmi  }

\title{ Poisson Kernels and Hilbert Transforms for Trigonometric Heckman–Opdam Polynomials of type $A_1$}
\date{ }

\begin{document}
 \maketitle
 \begin{center}
 Department of Mathematics, Faculty of Sciences. Taibah university 
Al-madinah Al-Munawarah. Arabia Saudi.\\
E-mail: bamria@taibahu.edu.sa, aguesmi@taibahu.edu.sa
\end{center}
 \begin{abstract}
 In this paper, we investigate the  trigonometric Heckman–Opdam polynomials of type $A_1$. We establish connections with ultraspherical polynomials and derive an explicit expression for the associated Poisson kernel. Using the product formula, we introduce a natural convolution structure and develop a theory of fractional integrals associated with these polynomials. We also define a generalized Hilbert transform in the framework of the Cherednik operator and prove its boundedness on $L^p$-spaces.
 This work provides an alternative perspective on the approach of B. Muckenhoupt and  E.M. Stein \cite{MS}.

 \end{abstract}
 \section{Introduction}
Fourier series are a fundamental tool in the analysis of periodic functions and constitute one of the central foundations of harmonic analysis. Over the years, numerous generalizations of the classical Fourier expansion have been developed, many of which arise naturally within the theory of special functions and orthogonal polynomials. These generalized Fourier series typically correspond to expansions over orthonormal bases of suitably weighted 
$L^2$ -spaces.
In this work, we focus on the Fourier expansions associated with the non-symmetric trigonometric Heckman–Opdam polynomials of type $A_1$. These polynomials form a remarkable family of orthogonal functions depending on a non-negative multiplicity parameter $k$. When $k=0$
 they reduce precisely to the classical exponential functions $\{e^{inx}, n\in \mathbb{Z}\}$, and the resulting generalized Fourier series becomes the standard Fourier series.
The non-symmetric Heckman–Opdam polynomials were introduced by Opdam, see \cite{O1, O2}, as eigenfunctions of the Cherednik (or Dunkl–Cherednik) operators attached to a given root system. For $A_n$-root system these polynomials exhibit deep connections with non-symmetric Jack polynomials \cite{ba1,Ba} and arise naturally in representation theory, the theory of special functions, and generalized harmonic analysis.
The purpose of this paper is to study the corresponding Fourier analysis associated with the Heckman–Opdam polynomials $\{E_n^k(ix). \quad n\in \mathbb{Z}\}$, which forms an orthogonal basis of the weighted space $L^2([-\pi,\pi], |sin x|^{2k}dx)$.
 In particular, we investigate structural identities, differential–difference equations, norm formulas, and the construction of the Poisson kernel associated with this system. This kernel generalizes the classical Poisson kernel on the unit circle and constitutes a key ingredient in our setting,  where it is used to express the integral kernel of an associated Hilbert transform. This operator is introduced in analogy with the classical Hilbert transform,
we establish a Calderón–Zygmund type theory  for the integral kernel  obtaining pointwise estimates and smoothness conditions that ensure boundedness on 
$L^p([-\pi,\pi], |sin x|^{2k}dx)$, $1<p<\infty$.
\section{Non-symmmetric Heckmann-Opdam polynomials  of type $A_1$ }
  We begin by reviewing and presenting several results concerning the non-symmetric Heckman–Opdam polynomials of type $A_1$. For the theory of   Heckmann-Opdam polynomials associated to general root systems,  we refer to   \cite{O1,O2}.
\par Let $k\geq 0$, the Cherednick operator of type $A_1$ is given by
\begin{eqnarray}\label{T}
T^{k}(f)(x)=\dfrac{df}{dx}(x)+ 2k\;\frac{ f(x)-f(-x)}{1-e^{-2x}}-   k f(x), \qquad f\in C^1(\mathbb{R}).
 \end{eqnarray}
 We define the weight function
$$dm_k(x)=|\sin x |^{2k} dx  $$
and the inner product on $L^2([-\pi,\pi],dm_k(x))$
$$(f,g)_k=\int_{-\pi}^{\pi}f(x)\overline{g(x)}dm_k(x),$$
with its associate norm $\|.\|_{2,k}$. Introduce a partial ordering $\triangleleft$ on $\mathbb{Z}$ as follows:
\begin{equation*}
  j\triangleleft n  \Leftrightarrow \left\{
    \begin{array}{ll}
    |j| < |n|\;\text{and} \;|n|-|j|\in 2\mathbb{Z}^+   , & \hbox{} \\
    or   , & \hbox{} \\
      |j| = |n|\;\text{and} \; n <j . & \hbox{.}
    \end{array}
  \right.
\end{equation*}
The non-symmetric  Heckman-Opdam polynomials $E_n^k$, $n\in \mathbb{Z}$ are characterized by the following conditions:
\begin{eqnarray}\label{def}
 (a)&&    E_n^{k}(z)=\textrm{e}^{nz}+ \sum_{j\triangleleft n }c_{n,j}\; \textrm{e}^{jz},\qquad z\in \mathbb{C}\qquad\qquad\qquad\qquad\qquad\qquad\qquad\qquad
 \\ (b) && (E_n^{k}(\textrm{i}x),\textrm{e}^{\textrm{i}jx})_{k}=0,\quad \text{for any } \; j\triangleleft n.
\end{eqnarray}
  It follows from Corollary 5.2 and
Proposition 6.1 in \cite{Sah} that the coefficients  $c_{n,j}$  are all real and non-negative.
The  polynomials $E_n^k$ diagonalize simultaneously the  Cherednik operators. In particular, they satisfy 
\begin{equation}\label{TE}
  T^{k} ( E_n^{k})= n_k \; E_n^{k}, \qquad n\in \mathbb{Z},
\end{equation}
 where $ n_k=n+k $ if $n >0$ and  $n_k=n-k$ if $n\leq 0$.
We observe that,
\begin{equation}\label{13}
 E_0^k(x)=1,\quad \text{and}\quad E_1^k(x)=e^x.
\end{equation}
 Moreover, the family of trigonometric polynomials
$\{ E_n^{k}(ix), \;n\in \mathbb{Z}\}$    is an orthogonal basis of 
  $L^2([-\pi,\pi],dm_k(x))$.\\
 \par The Heckman-Opdam-Jacobi  polynomials  $P_n^k$, $n\in \mathbb{Z}_+$, are defined by
  $$P_n^{k}(x) =\frac{1}{2} \Big(E_n^{k}(x)+E_n^{k}(-x)\Big),$$
and they satisfy 
$$ (T^k)^2P_n^k=(n+k)^2 P_n^k.$$
In particular,  $P_n^k$ is an  eigenfunction of the operator  $L_k$, which coincides with the even part of $(T^k)^2$, given by 
$$L_k(f)=f''+2k\; \frac{1+e^{-2x}}{1-e^{-2x}}\;f'(x)+k^2f(x)$$
The Heckman-Opdam-Jacobi  polynomials are expressed via the hypergeometric function $_2F_1$
as
\begin{equation}\label{P}
P_n^k(x)=P_n^k(0)\;_2F_1(n+2k,-n,k+1/2,-\sinh^2(x/2))
\end{equation}
where, for all $n\in \mathbb{Z}^+$,
$$ P_n^k(0)=E_n^{k}(0)= \frac{\Gamma(k+1)\Gamma(n+2k)}{\Gamma(2k+1)\Gamma(n+k)}=\frac{\Gamma(k)\Gamma(n+2k)}{2\Gamma(2k)\Gamma(n+k)}.$$
  An important identity that can follows from  (\ref{P}) is
\begin{equation}\label{pn}
\frac{dP_n^k}{dx}(x)= 2 nP_{n-1}^{k+1}(x)\sinh x, \quad n\geq 1.
\end{equation} 
Further,  making use of the identities:
$$\frac{dE_n^k}{dx} (x)+2k \frac{E_n^k(x)-E_n^k(-x)}{1-e^{-2ix}}=(n+2k)E_n^k(x)$$
$$\frac{dE_n^k}{dx}-2k \frac{E_n^k(x)-E_n^k(-x)}{1-e^{2x}}=(n+2k)E_n^k(-x)$$
and subtracting the second  from the first we obtain  that 
$$\frac{dP_n^k}{dx}(x)=n\;\dfrac{E_n^k(x)-E_n^k(-x)}{2}.$$
 This yields, in view of (\ref{pn}), the representation formula:
\begin{equation}\label{jacobi}
   E_n^k(x)= P_{n}^k( x)+2\sinh x \;P_{n-1}^{k+1}(x),\quad n\geq 1.
\end{equation}
 \par It should be noted that the non-symmetric  Heckman-Opdam polynomials $E_n^k$ are closely  related to non-symmetric
Jack polynomials,  see \cite{Sah}.
If  $\lambda=(\lambda_1,\lambda_2)\in \mathbb{Z}\times \mathbb{Z} $ and $\mathcal{I}_{\lambda }^k$ is the correspondent non-symmetric
Jack polynomial then
$$E_{\lambda_2-\lambda_1}^k= \mathcal{I}_{\lambda}^k(e^{-x},e^{x}).$$
 \begin{prop}
 For all $n\in \mathbb{Z}$, we have
\begin{equation}\label{12}
   E_{n+1}^k(x)=e^xE_{-n}^k(-x).
\end{equation}
\end{prop}
\begin{proof}
In view of (\ref{13}), the identity (\ref{12}) holds for     for $n= -1,0$. Let $n\in \mathbb{Z}$, $n\neq -1,0$. Put
 $H_n^k(x)= e^xE_{-n}^k(-x)$. It is  enough to check that $ T^k(H_n)=(n+1)_kH_n$. We have
\begin{eqnarray*}
&&T^k(H_n)(x)
\\ &&\quad =e^xE_{-n}^k(-x)-e^x(E_{-n}^k)'(-x)+2k \; \frac{e^xE_{-n}^k(-x)-e^{-x}E_{-n}^k(x)}{1-e^{-2x}}-ke^xE_{-n}^k(-x)
\\&&\quad= -e^x \left((E_{-n}^k)'(-x)+2k\;\frac{E_{-n}^k(-x)-E_{-n}^k(x)}{1-e^{2x}} -k E_{-n}^k(-x) \right)+  e^xE_{-n}^k(-x)
\\&&\quad= ( 1- (-n)_k)e^xE_{-n}(-x)=(n+1)_ke^xE_{-n}(-x)=(n+1)_kH_n(x).
\end{eqnarray*}
Comparing the leading coefficients yields (\ref{12}). 
\end{proof}
\begin{prop}
 For $n \in \mathbb{Z}^+$, $n\neq 0$, we have
\begin{equation}\label{45}
 E_{-n}^k(x)=E_{n}^k(-x) +\frac{k}{n+k}E_n^k(x).
\end{equation}
  \end{prop}
\begin{proof}
Observe that  for $n\geq 1$ we have $T^k(E_n(-x))=-2kE_n(x)- (n+k)E_n(-x)$. Then
\begin{eqnarray*}
  T^k\left( E_n^k(-x)+\frac{k}{n+k}E_n^k(x)\right)&=&-2kE_{n}(x)- (n+k)E_n(-x)+ kE_{n}(x)
\\&=&-kE_{n}(x)-(n+k)E_n^k(-x)
\\&=&-(n+k)\left(E_n^k(-x)+\frac{k}{n+k}E_n^k(x)\right).
 \end{eqnarray*}
Equality  (\ref{45}) follows by comparing the highest coefficient.
\end{proof}
  \begin{prop}
  For all $n  \in \mathbb{Z}^+$, we have
  \begin{equation}\label{Ga}
   \|E_{n+1}^k\|_k^2=\|E_{-n}^k\|_k^2=  \pi2^{1-2k} n!\; \frac{\Gamma(n+2k+1)}{\Gamma(n+k+1)^2}. 
  \end{equation}
  Here for simplicity, we write  $\|E_{j}^k\|_k^2$ instead of $\|E_{j}^k((ix)\|_k^2$.
\end{prop}
\begin{proof}
In view of  (\ref{45}), we write
$$  E_n^k(-x)= E_{-n}^k(x)-\frac{k}{n+k}E_n^k(x), \quad n\geq 1\quad.$$
  Since $E_n^k$ and $E_{-n}^k$ are orthogonal, it follows that
$$\|E_n^k\|_{2,k}^2= \|E_{-n}^k\|_{2,k}^2 +\frac{k^2}{(n+k)^2}\|E_n^k\|_{2,k}^2$$
 and hence
 \begin{equation}\label{69}
\|E_{n+1}^k\|_{2,k}^2= \|E_{-n}^k\|_{2,k}^2 =
\frac{n(n+2k)}{(n+k)^2} \|E_n^k\|_{2,k}^2 .
 \end{equation}
Moreover, one easily checks that
 $$ \|E_1^k\|_{2,k}^2=\|E_{0}^k\|_{2,k}^2=   \pi2^{1-2k}\frac{\Gamma(2k+1)}{\Gamma(k+1)^2}=\dfrac{2\sqrt{\pi}\Gamma(k+\dfrac{1}{2})}{\Gamma(k+1)}.$$
The formula (\ref{Ga}) then follows.
\end{proof}
 \par In conclusion,  the explicit expressions of the Heckman–Opdam polynomials $E_n^k$ are given by:
$E_0^k = 1$, and for all $n\geq 1$,
$$
\begin{cases}
\displaystyle
E_n^k(x)
= P_n^k(x) + 2\sinh(x)\, P_{n-1}^{\,k+1}(x), \\[2mm]
\displaystyle
E_{-n}^k(x)
= \frac{n+2k}{\,n+k\,}\, P_n^k(x)
- \frac{2n}{\,n+k\,}\, \sinh(x)\, P_{\,n-1}^{\,k+1}(x).
\end{cases}
$$
\section{Poisson kernel}
 The Heckman–Opdam–Jacobi polynomials can be expressed in terms of ultraspherical polynomials,  
 \begin{equation}\label{pc}
 P_n^k(ix)=\dfrac{P_n^k(0)}{C_n^k(1)}C_n^k(cosx)=    \dfrac{\Gamma(k)n!}{2\Gamma(n+k)}C_n^k(cosx).
 \end{equation}
Here, the ultraspherical polynomials  $C_n^k$   are defined via  their generating function.
$$\sum_{n=0}^{+\infty}r^n C_n^k(t)=\frac{1}{(1-2tr+r^2)^k}, \qquad |r|<1.$$
We will use the following identities for the ultraspherical polynomials:
 $$C_n^k(1)=\frac{\Gamma(n+2k)}{n!\Gamma(2k)}, \quad \|C_n^k\|_{2,k}^2=\pi 2^{1-2k}\dfrac{\Gamma(n+2k)}{n!(n+k)\Gamma(k)^2}\quad\text{and}\quad \frac{d}{dx} C_n^k(x)=2k C_{n-1}^{k+1}(x)$$
  (See Szegö, (\cite{G})).  
  Let $\gamma_n^k=\|E_n^k\|_{2,k}^{-2} $. We define the Poisson kernel $ P_k(r,x,y)$, 
   for $0\leq r <1$ and $x,y\in [-\pi,\pi]$, by
    \begin{eqnarray*}
     P_k(r,x,y)&=&\sum_{n=-\infty}^{+\infty}\gamma_n^kr^{|n|+k}E_n^k(ix)E_n^k(-iy) \\&=&
  \sum_{n=0}^{+\infty}\gamma_n^kr^{n+k}E_n^k(ix)E_n^k(-iy)+\sum_{n=1}^{+\infty}
  \gamma_{-n}^k r^{n+k}E_{-n}^k(ix)E_{-n}^k(-iy).
    \end{eqnarray*}
 This kernel is real. In fact, using (\ref{45}) we may write
 $$ \sum_{n=1}^{+\infty}\gamma_{-n}^k r^{n+k}E_{-n}^k(ix)E_{-n}^k(-iy)=
  \sum_{n=1}^{+\infty}\gamma_{-n}^k \frac{k}{n+k}r^{n+k}\Big\{E_{n}^k(ix)E_{n}^k(iy)+E_{n}^k(-ix)E_{n}^k(-iy)\Big\} $$
  $$+  \sum_{n=1}^{+\infty}\gamma_{-n}^k r^{n+k}E_{n}^k(-ix)E_{n}^k(iy) +  
  \sum_{n=1}^{+\infty}\gamma_{-n}^k \frac{k^2}{(n+k)^2}r^{n+k}E_{n}^k(ix)E_{n}^k(-iy)$$
By virtue of (\ref{69}), we have 
 $$ \sum_{n=1}^{+\infty}\gamma_{-n}^k \frac{k^2}{(n+k)^2}r^{n+k}
 E_{n}^k(ix)E_{n}^k(-iy)
 =\sum_{n=1}^{+\infty}\gamma_{-n}^k r^{n+k}
 E_{n}^k(ix)E_{n}^k(-iy)-\sum_{n=1}^{+\infty}\gamma_{n}^kr^{n+k}E_{n}^k(ix)E_{n}^k(-iy)$$
Collecting terms gives
  $$ P_k(r,x,y)= \gamma_0^k r^{k} +2\sum_{n=1}^{+\infty}\gamma_{-n}^k \frac{k}{n+k}r^{n+k}\Re\Big\{E_{n}^k(ix)E_{n}^k(iy)\Big\} +2 \sum_{n=1}^{+\infty}\gamma_{-n}^k r^{n+k}
 \Re\Big\{E_{n}^k(ix)E_{n}^k(-iy)\Big\}.$$ 
 \par  Using the identities relating the non-symmetric polynomials $E_n^k$ 
and the symmetric  polynomials $P_n^k$, the Poisson kernel can be written in the form
 \begin{eqnarray*}
     P_k(r,x,y)&=& \gamma_0^k r^{k}+2\sum_{n=1}^{+\infty} \frac{n+k}{n}\gamma_{n}^kr^{n+k}P_n^k(ix)P_n^k(iy)\\&&\qquad\qquad\qquad\qquad+8\sin x\sin y\sum_{n=1}^{+\infty}\frac{n+k}{n+2k}\gamma_{n}^kr^{n+k}P_{n-1}^{k+1}(ix)P_{n-1}^{k+1}(iy)
     \\&=&\frac{1}{2}\sum_{n=0}^{+\infty}r^{n+k}\frac{C_n^k(\cos x)C_n^k(\cos y)}{\|C_n^k\|^2}+\frac{1}{2}\sin x\sin y\sum_{n=0}^{+\infty}r^{n+k+1}\frac{C_n^{k+1}(\cos x)C_n^{k+1}(\cos y)}{\|C_n^{k+1}\|^2}.
    \end{eqnarray*}
We next invoke the following fundamental identity (see for instance (2.12) of \cite{MS}):  
 \begin{equation}
 \sum_{n=0}^{+\infty}r^{n}\frac{C_n^k(\cos x)C_n^k(\cos y)}{\|C_n^k\|^2}= \frac{k}{\pi}(1-r^2)\int_{-1}^1\dfrac{(1-u^2 )^{k-1} du}{(1-2r(\cos xcos y+u\sin x\sin y)+r^2 )^{k+1}}.
\end{equation}  
 We obtain
\begin{eqnarray} \label{33}
     P_k(r,x,y)&=& \frac{k}{2\pi}(1-r^2)r^{k}\int_{-1}^1\dfrac{(1-u^2 )^{k-1} du}{(1-2r(\cos xcos y+u\sin x\sin y)+r^2 )^{k+1}}
 \nonumber    \\&+& \frac{k+1}{2\pi}(1-r^2)r^{k+1}\sin x\sin y\int_{-1}^1\dfrac{(1-u^2 )^{k} du}{(1-2r(\cos xcos y+u\sin x\sin y)+r^2 )^{k+2}}.\nonumber\\&&
    \end{eqnarray}
  By performing an integration by parts on the second integral, where we have 
  $$\frac{k+1}{2\pi}(1-r^2)r^{k+1}\sin x\sin y\int_{-1}^1\dfrac{(1-u^2 )^{k} du}{(1-2r(\cos xcos y+u\sin x\sin y)+r^2 )^{k+2}}$$
$$ = \frac{k}{2\pi}(1-r^2)r^{k}\int_{-1}^1\dfrac{u(1-u^2 )^{k-1} du}{(1-2r(\cos xcos y+u\sin x\sin y)+r^2 )^{k+1}}$$
we arrive at the compact  representation
\begin{eqnarray*}
     P_k(r,x,y)= \frac{k}{2\pi}(1-r^2)r^{k}\int_{-1}^1\dfrac{(1+u)(1-u^2 )^{k-1} du}{(1-2r(\cos xcos y+u\sin x\sin y)+r^2 )^{k+1}}.
     \end{eqnarray*}
We  state this result in the following theorem:
\begin{thm}
 The Poisson kernel has the following expansion
\begin{eqnarray}
     P_k(r,x,y)= \frac{k}{2\pi}(1-r^2)r^{k}\int_{-1}^1\dfrac{(1+u)(1-u^2 )^{k-1} du}{(1-2r(\cos xcos y+u\sin x\sin y)+r^2 )^{k+1}}.
     \end{eqnarray}
\end{thm}
To confirm consistency, we test the formula for $k=0$
and show that it coincides with the classical Poisson kernel. 
Define 
$$h_k(u)=\dfrac{u^{-1}}{(1-2r(\cos xcos y+u\sin x\sin y)+r^2)^{k+1}}$$
In view of (\ref{33})
$$\frac{k}{2\pi}(1-r^2)r^{k}\int_{-1}^1\dfrac{(1-u^2 )^{k-1} du}{(1-2r(\cos xcos y+u\sin x\sin y)+r^2 )^{k+1}}$$
$$=-\frac{1}{4\pi}(1-r^2)r^{k}
\lim_{\varepsilon\rightarrow 0^+}\int_{\varepsilon\leq |u|\leq 1}-2ku(1-u^2 )^{k-1}h(u) du $$
$$=-\frac{1}{4\pi}(1-r^2)r^{k}\lim_{\varepsilon\rightarrow 0^+}
\left\{(1-\varepsilon^2)^kh(\varepsilon)-(1-\varepsilon^2)^kh(-\varepsilon)+\int_{\varepsilon\leq |u|\leq 1}(1-u^2)^k \frac{d h(u)}{du}\right\}$$  
Setting  $k=0$, this simplifies to
$$ h_k(1)-h_k(-1)=\frac{1}{4\pi}(1-r^2) \left\{\frac{1}{1-2r(\cos (x-y)+r^2}+
\dfrac{1}{1-2r(\cos (x+y)+r^2} \right\}$$
On the other hand, we have
$$ \frac{1}{4\pi}(1-r^2)\int_{-1}^1\dfrac{2 r\sin x\sin ydu}{(1-2r(\cos xcos y+u\sin x\sin y)+r^2 )^{2}}
$$
$$=\frac{1}{4\pi}(1-r^2) \left\{\frac{1}{1-2r(\cos (x-y)+r^2}-
\dfrac{1}{1-2r(\cos (x+y)+r^2} \right\}.$$
From the previous computations, we obtain the explicit expression of the Poisson for $k=0$
$$  P_0(r,x,y)=\frac{1}{2\pi}\; \frac{1-r^2 }{1-2r(\cos (x-y)+r^2}$$
which is exactly the classical Poisson kernel. 
\par The Poisson kernel  $P_k(r, x,y)$
enjoys the following basic  properties:
\begin{eqnarray*}
P_k(r, x,y) \geq 0,\;\; P(r,x,y)= P(r, y,x )\;\text{and}\; 
\int_{\pi}^{\pi}P_k(r, x,y) dm_k(x)=\pi 2^{1-2k}\dfrac{\Gamma(2k+1)}{\Gamma(k+1)^2}r^k. 
\end{eqnarray*}
For a function $$f(x)=\sum_{-\infty}^{+\infty} a_n E_n^k(ix),$$ 
we define its Poisson integral by
$$f(r,x)=\sum_{-\infty}^{+\infty} a_nr^{|n|+k}E_n^k(ix),   \qquad 0\leq r<1, \; x\in[-\pi,\pi]$$ 
In particular, the Poisson integral admits the representation
$$f(r,x)= \int_{-\pi}^{\pi}P_k(r,x,y)f(y)dm_k(y)$$ 
In what follows, we present several classical results for   the Poisson integral $f(r,x)$. The arguments are essentially the same as those in \cite{MS}, and for brevity, we omit the proofs. These results illustrate the fundamental properties of the Poisson semigroup in the context of the measure $dm_k(x)$.
\begin{thm}
Let $f\in L^p(dm_k)$. Then
\begin{itemize}
\item[(a)]  $\|f(r,x)\|_{p,k}\leq \|f\|_{p,k}$,  $1\leq p\leq \infty$. 
\item[(b)]  $\|f(r,x)-f\|_{p,k} \rightarrow 0$ as $r\rightarrow 1$, $1\leq p< \infty$. 
\item[(c)]  $\lim_{r\rightarrow1}f(r,x)=f $ almost everywhere, $1\leq p\leq \infty$.  
\item[(d)]  $\|\sup_{r<1}f(r,x)\|_{p,k}\leq \|f\|_{p,k}$, $1< p\leq \infty$. 
\end{itemize}
\end{thm}
   
\section{Convolution structure and fractional integration.}

In this section, we introduce a notion of convolution product associated with the  Heckman–Opdam polynomials  $E_n^k(ix)$ which arises from the product formula expressing the product of two polynomials of the same degree as an integral kernel.
\par Our starting point is the product formula satisfied by the ultraspherical polynomials.
(see \cite{G}).
For $x,y\in[-\pi,\pi]$, we have 
$$\frac{C_n^k(\cos x)}{C_n^k(1)}\frac{C_n^k(\cos y)}{C_n^k(1)}=\dfrac{\Gamma(k+1/2)}{\Gamma(k)\sqrt{\pi}}\int_{-1}^1\dfrac{C_n^k(\cos x\cos y+u\sin x\sin y)}{C_n^k(1)}(1-u^2)^{k-1} du.$$
 Invoking the relation between $P_n^k$ and $C_n^k$
 provided by (\ref{pc}) we rewrite the above expression accordingly, 
 \begin{equation}\label{ppw}
 \frac{P_n^k(ix)}{P_n^k(0)}\frac{P_n^k(iy)}{P_n^k(0)}= \dfrac{\Gamma(k+1/2)}{\Gamma(k)\sqrt{\pi}}\int_{-1}^1\frac{P_n^k(i\arccos(\cos x\cos y+u\sin x\sin y))}{P_n^k(0)} (1-u^2)^{k-1} du.
 \end{equation}
If  $\sin x\sin y\neq 0$  then  making the substitution $$z=\arccos\Big(\cos x\cos y+u\sin x\sin y\Big),$$
 the integral becomes
 \begin{equation}\label{pw}
 \frac{P_n^k(ix)}{P_n^k(0)}\frac{P_n^k(iy)}{P_n^k(0)}= \int_{-\pi}^\pi\frac{P_n^k(iz)}{P_n^k(0)} W_k(x,y,z) dm_k(z). 
 \end{equation}
where $W_k$ is given by on $[-\pi,\pi]^3$ by
\begin{eqnarray*}
 && W_k(x,y,z)\\ &&\qquad =\dfrac{\Gamma(k+1/2)}{2\Gamma(k)\sqrt{\pi}}
 \frac{\Big\{ (\cos z-\cos(x+y))(\cos(x-y)-cos z\Big\}^{k-1}}
 {\Big|\sin x\sin y\sin z\Big|^{2k-1}}\\
 & &\qquad=4^{k-1}\dfrac{\Gamma(k+1/2)}{2\Gamma(k)\sqrt{\pi}}\\&&
 \frac{\Big\{ \sin\left( \dfrac{x+y+z}{2}\right)\sin\left( \dfrac{x+y-z}{2}\right)
 \sin\left( \dfrac{x-y+z}{2}\right)\sin\left( \dfrac{-x+y+z}{2}\right)\Big\}^{k-1}}
 {\Big|\sin x\sin y\sin z\Big|^{2k-1}} 
\end{eqnarray*}
 if $\cos (|x|+|y|)< \cos z< cos (|x|-|y|)$
 and $W_k(x,y,z)=0$  otherwise.
 
 \begin{lem}\label{03}
Let  $x,y,z\in [0,\pi]$. Then 
 $$\cos(x+y)\leq \cos z\leq \cos (x-y) 
  \Longleftrightarrow \cos(z+y)\leq \cos x\leq \cos (z-y).$$
\end{lem}
\begin{proof}
Since the function $\arccos$ is decreasing,  we have 
$$\cos(x+y)\leq \cos z\leq \cos (x-y) 
 \Longleftrightarrow |x-y|\leq z\leq \arccos(\cos(x+y))=\pi-|x+y-\pi|.$$
Observe that 
 $$\pi-|x+y-\pi|=\min(x+y,2\pi-(x+y)).$$
From this, we deduce
$$|x-y|\leq z\leq x+y$$
 which implies that $|y-z|\leq x\leq y+z$.
Moreover, since $x+y+z\leq 2\pi$, it follows that
 $x\leq 2\pi-(z+y)$.
Therefore, Combining the two upper bounds for $x$, we obtain
$$|z-y|\leq x\leq \min(z+y,2\pi-(z+y))=\pi-|z+y-\pi|.$$
Now, as $\cos$ is decreasing  on $[0,\pi]$, we get that 
$$\cos(z+y)\leq \cos x\leq \cos (z-y).$$
This concludes the proof..
 \end{proof}
 The following proprieties of $W_k$ are immediate.
 \begin{prop}
   For all  $x,y,z\in[-\pi,\pi]$ the following hold:
 \begin{itemize}
\item[(i)] Non-negativity: $W_k(x,y,z)\geq 0$.
\item[(ii)] Symmetry: $W_k(x,y,z)=W_k(y,x,z)=W_k(x,z,y)$.
\item[(iii)]Normalization: For all $y,z \in[-\pi,\pi]$ ,$\displaystyle{\int_{-\pi}^\pi}W_k(y,x,z) dm_k(x)=1$.
 \end{itemize}
 \end{prop}
\par Next, for $n\in \mathbb{Z}$ we let
$$\mathcal{E}_n^{k}(ix)=\frac{E_n^k(ix)}{E_n^k(0)}.$$
In what follows, we shall prove that $\mathcal{E}_n^{k}$ satisfies the following product formula.
\begin{thm} For $x,y\in [-\pi,\pi]$, we have 
 $$\mathcal{E}_n^k(ix)\mathcal{E}_n^k(iy)= \int_{-\pi}^{\pi}\mathcal{E}_n^k(iz) d\mu_k^{(x,y)}(z)$$
with
  $$
d\mu_k^{(x,y)}(z)=
 \begin{cases}
 \mathcal{ W}_k(x,y,z)dm_k(z) &\text{if}\; \sin x\sin y\neq 0,
 \\ d\delta_{x}   &\text{if}\; y=0,
 \\ d\delta_{x\pm \pi}   &\text{if}\;  y=\pm \pi,
 \\\ d \delta_{y} &\text{if}\;  x=0
  \\ d\delta_{y\pm\pi}   &\text{if}\;  x=\pm \pi,
 \end{cases}
 $$
 where $\delta$ denotes is the Dirac measure and 
 \begin{eqnarray*}
 \mathcal{ W}_k(x,y,z)=&&4 \exp\Big(\dfrac{1}{2}i(x+y-z)\Big)\\&&
  \dfrac{\sin\left(\dfrac{x+y+z}{2}\right)\sin\left(\dfrac{-x+y+z}{2}\right)\sin\left(\dfrac{x-y+z}{2}\right)}{\sin x\sin y\sin z}\; W_k(x,y,z).
 \end{eqnarray*}
\end{thm}
\begin{proof}
We  consider  only the case  $n\geq 0$, since for  $n<0$,  the formula can be derived using  (\ref{12}). For convenience we decompose $\mathcal{E}_n^{k}(ix)$ into its even and odd parts
$$\mathcal{E}_n^{k}(ix)=\mathcal{E}_{n,e}^{k}(ix)+\mathcal{E}_{n,o}^{k}(ix).$$
Suppose   $\sin x\sin y\neq 0$. From (\ref{pw}) we have
\begin{eqnarray*}
\mathcal{E}_{n,e}^{k}(ix)\mathcal{E}_{n,e}^{k}(iy)= \frac{P_n^k(ix)}{P_n^k(0)}\frac{P_n^k(iy)}{P_n^k(0)}&=& \int_{-\pi}^\pi\frac{P_n^k(iz)}{P_n^k(0)} \;W_k(x,y,z) dm_k(z) \\&=&\int_{-\pi}^\pi\mathcal{E}_{n}^{k}(ix)\;W_k(x,y,z) dm_k(z).
\end{eqnarray*} 
 Similarly, in view of (\ref{ppw}), we write  
\begin{eqnarray*}
&&\mathcal{E}_{n,o}^{k}(ix)\mathcal{E}_{n,o}^{k}(iy)=-4\sin x\sin y
 \frac{P_{n-1}^{k+1}(ix)}{P_{n}^{k}(0)}\frac{P_{n-1}^{k+1}(iy)}{P_{n}^{k}(0)}
\\&=&
-\frac{\Gamma(k+1/2)}{\Gamma(k+1)\sqrt{\pi}}   (n+2k)\int_{-1}^1  
\frac{P_{n-1}^{k+1}(i\arccos(\cos x\cos y+u\sin x\sin y)) }{P_{n}^{k}(0)}
\sin x\sin y(1-u^2)^{k} du
\\&&=-\frac{\Gamma(k+1/2)}{\Gamma(k+1)\sqrt{\pi}} n\int_{-1}^1  
\frac{P_{n-1}^{k+1}(i\arccos(\cos x\cos y+u\sin x\sin y)) }{P_{n}^{k}(0)}
\sin x\sin y(1-u^2)^{k} du
\\&&-\frac{2\Gamma(k+1/2)}{\Gamma(k)\sqrt{\pi}}  \int_{-1}^1  
\frac{P_{n-1}^{k+1}(i\arccos(\cos x\cos y+u\sin x\sin y)) }{P_{n}^{k}(0)}
\sin x\sin y(1-u^2)^{k} du.
\\&& =I_{k,1}+I_{k,2}.
\end{eqnarray*}
Observe that when  
$ z=\arccos(\cos x\cos y+u\sin x\sin y)$
we have that
$$\sin x\sin y=-\sin z \dfrac{dz}{du}.$$
Then integration by parts in $I_{k,1}$ together with the (\ref{pn}) give 
\begin{eqnarray*}
I_{k,1}&=&-\frac{\Gamma(k+1/2)}{\Gamma(k)\sqrt{\pi}} \int_{-1}^1 
 \frac{P_n^k(i\arccos(\cos x\cos y+u\sin x\sin y))}{P_n^k(0)}u(1-u^2)^{k-1}
\\&=&-\int_{-\pi}^{\pi}
\mathcal{E}_n^k(iz) \dfrac{\cos z-cos x\cos y}{\sin x\sin y}W_k(x,y,z)dm_k(z).
\end{eqnarray*}
For  $I_{k,2}$ we have 
\begin{eqnarray*}
I_{k,2}&=&-i\int_{-\pi}^{\pi} 
\mathcal{E}_n^k(iz)
\frac{ (\cos z-\cos(x+y))(\cos(x-y)-cos z)}
 {|\sin x\sin y| \sin z} W_k(x,y,z)dm_k(z)
  \\&=&4i\int_{-\pi}^{\pi} 
\mathcal{E}_n^k(iz)
\frac{ \sin(\frac{x+y+z}{2})  \sin(\frac{x-y+z}{2}) \sin(\frac{-x+y+z}{2}) \sin(\frac{x+y-z}{2})}
 {\sin x\sin y \sin z} W_k(x,y,z)dm_k(z).
\end{eqnarray*}
On the other hand, using the identity
$$-\frac{1}{n}\left(\dfrac{d}{dx}+\dfrac{d}{dy}\right)P_{n}^{k}(i\arccos(\cos x\cos y+u\sin x\sin y)$$
$$=(1-u)\sin(x+y)P_{n-1}^{k+1}(i\arccos(\cos x\cos y+u\sin x\sin y)$$
we obtain
\begin{eqnarray*}
&&\mathcal{E}_{n,o}^{k}(ix)\mathcal{E}_{n,e}^{k}(iy)\\&=&
 2i \sin x\frac{P_{n-1}^{k+1}(ix)}{P_{n}^{k}(0)}\frac{P_{n}^{k}(iy)}{P_{n}^{k}(0)}+ 2i\sin y \frac{P_{n-1}^{k+1}(ix)}{P_{n}^{k}(0)}\frac{P_{n}^{k}(iy)}{P_{n}^{k}(0)}
 \\&=&-\frac{i}{n}\left(\dfrac{d}{dx}+\dfrac{d}{dy}\right)\Big\{ \frac{P_{n}^{k}(ix)}{P_{n}^{k}(0)}\frac{P_{n}^{k}(iy)}{P_{n}^{k}(0)}\Big\}\\&=&-
i\dfrac{\Gamma(k+1/2)}{\Gamma(k)\sqrt{\pi}}
\int_{-1}^1\frac{P_{n-1}^{k+1}(i\arccos(\cos x\cos y+u\sin x\sin y))}{P_{n}^{k}(0)} \sin(x+y)(1-u) (1-u^2)^{k-1} du\\&=&
-\int_{-\pi}^{\pi} 
\mathcal{E}_n^k(iz)\dfrac{\sin (x+y)(\cos(x-y)-\cos z)}{\sin x\sin y \sin z}
  W_k(x,y,z)dz.
\end{eqnarray*} 
Summing all the terms yields the final expression for the kernel $\mathcal{W}_k(x,y,z)$, since 
\begin{eqnarray*}
&&1+ \dfrac{\cos z-cos x\cos y}{\sin x\sin y}-\dfrac{\sin (x+y)(\cos(x-y)-\cos z)}{\sin x\sin y \sin z}\\&&\qquad \qquad\qquad\qquad\qquad\qquad=\frac{(\cos(x-y)-\cos z)(s\sin z+sin(x+y)}{\sin x\sin y\sin z}\\&&\qquad\qquad\qquad\qquad\qquad\qquad=4\dfrac{ \sin(\frac{x+y+z}{2})  \sin(\frac{x-y+z}{2}) \sin(\frac{-x+y+z}{2}) \sin(\frac{x+y-z}{2})}
 {\sin x\sin y \sin z}.
\end{eqnarray*}
This completes the derivation of the desired formula.
 \end{proof}
 We point out the following remark, as it will be useful in subsequent arguments.
  From the following identity 
\begin{eqnarray*}
\dfrac{\sin (x+y)(\cos z-\cos(x-y))}{\sin x\sin y \sin z} =
 \dfrac{cos x\cos z-\cos y}{\sin x\sin z}+\dfrac{cos y\cos z-\cos x}{\sin y\sin z}
\end{eqnarray*}
we may rewrite 	$\mathcal{ W}_k (x,y,z)$ in the form
\begin{eqnarray*}
 &&\mathcal{ W}_k(x,y,z)=4 \exp\Big(\dfrac{1}{2}i(x+y-z)\Big)\\&&
\left\{1-  \dfrac{cos x\cos y-\cos z}{\sin x\sin y}+  \dfrac{cos x\cos z-\cos y}{\sin x\sin z}+\dfrac{cos y\cos z-\cos x}{\sin y\sin z}\right\}\; W_k(x,y,z).
 \end{eqnarray*}
In view of the  lemma \ref{03} we can state the following inequality
$$|\mathcal{W}_k(x,y,z)|\leq 16 W_k (x,y,z).$$
\par  We now introduce the translation operator 
 on $L^2([-\pi,\pi],dm_k(x))$. Let $x\in[-\pi,\pi]$. 
 For $f\in L^2([-\pi,\pi],dm_k(x))$ with the expansion
 $$f(x)=\sum_{n\in \mathbb{Z}}a_n \mathcal{E}_n^k(ix)$$
the translation of $f$ by $x$ is defined by 
$$\tau_x(f)(y)=\sum_{n\in \mathbb{Z}}a_n \mathcal{E}_n^k(ix) \mathcal{E}_n^k(iy)=
\int_{-\pi}^\pi f(z)d\mu_k^{(x,y)}(z).$$
We list below some basic properties of the translation operator, their verification is straightforward.
 \begin{prop}
  Let $x\in[\pi,-\pi]$. The operator $\tau_x$ satisfy
  \begin{itemize}
  \item[(i)] $\tau_x(\mathcal{E}_n^k(iy))=\mathcal{E}_n^k(ix)\mathcal{E}_n^k(iy)$.
  \item[(ii)]$\tau_x(f)(y)=\tau_y(f)(x)$, for all $y\in[\pi,-\pi]$.
  \item[(iii)]$\|\tau_x(f)(y)\|_{2,k}\leq \|f\|_{2,k}$, 
  for all $f\in L^2([-\pi,\pi],dm_k(x))$.
  \item[(iv)]  $f\in L^2([-\pi,\pi],dm_k(x))\cap  L^1([-\pi,\pi],dm_k(x))$
   and  $x\in[\pi,-\pi]$.
  $$\int_{-\pi}^\pi \tau_x(f)(y)dm_k(y)=\int_{-\pi}^\pi f(y)dm_k(y).$$
 \item[(v)] The operator $\tau_x$ can be extended to all $L^p([-\pi,\pi],dm_k(x))$, $1\leq p\leq\infty,$ with 
  $$\|\tau_x(f)(y)\|_{p,k}\leq 16\|f\|_{p,k}.$$
 \end{itemize}
 \end{prop} 
 We define  the convolution product  $f\star_k g$ as  
$$f\star_k g(x)=\int_{-\pi}^{\pi}  \tau_xf(-y)g(y)dm_k(y).$$
This convolution product enjoys the following immediate properties (follow from standard classical arguments)
\begin{prop}
Let $f$ and $g$ be suitable functions. Then
\begin{itemize}
\item[(i)] $|f\star_k g|\leq |f|\star_k |g|$.
 \item[(ii)]$\mathcal{E}_n^k\star_k \mathcal{E}_m^k=\|\mathcal{E}_n\|^2_{2,k}\delta_{m,n}\mathcal{E}_n^k$.
\item[(iii)]
 $\|f\star_k g\|_{\infty,k}=\|f\star_k g\|_{\infty,k}\leq 16\|f\|_{1,k}\|g\|_{\infty,k}$.
\item[(iv)]  $\|f\star_k g\|_{1,k}\leq 16
\|f\|_{1,k}\|g\|_{1,k}$.
\item[(v)] Young's inequality: For $1\leq p,q,r\leq \infty$ with $\dfrac{1}{r}=\dfrac{1}{p}+\dfrac{1}{q}-1$, 
$$\|f\star_k g\|_{r,k}\leq 16 \|f\|_{p,k}\| g\|_{q,k}.$$
\end{itemize}
\end{prop}
The convolution product considered here satisfies the structural assumptions and properties of  Chapter 3 in  \cite{MS}. Consequently, according to Theorem 13 of \cite{MS}, the following result is valid in this context
 \begin{equation}
 \|f\star_k g\|_r\leq C(p,q)\|f\|_p\|g\|_{q,k}^*,\quad 1<p,q,r< \infty,\;\dfrac{1}{r}=\dfrac{1}{p}+\dfrac{1}{q}-1
 \end{equation}
where 
$$\|g\|_{q,k}^*=\sup_{E}\frac{\int_{-\pi}^{\pi}|g|\chi_{E}dm_k}{\|\chi_{E}\|_{q',k}},\qquad \frac{1}{q}+\frac{1}{q'}=1.$$
The supremum is taken over all measurable sets 
$E\subset[-\pi,\pi]$, $m_k(E)<\infty$, and  $\chi_{E}$ is the characteristic function of the set $E$.
\par We next introduce the fractional integral. It can be defined as follows. If
$$f= \sum_{n\in \mathbb{Z}}a_n \mathcal{E}_n^k$$
  then the fractional integral $ I_k^\alpha(f)$, $\alpha >0$, is given by
$$ I_k^\alpha(f)= \sum_{n\in \mathbb{Z},n\neq 0}|n|^{-\alpha}a_n \mathcal{E}_n^k.$$
\begin{thm}
 For $\alpha>0$ and  $1<p<\frac{2k+1}{\alpha}$, there exists a constant $C=C(\alpha,p)$ such that
$$\|I_k^\alpha(f\|_{q,k}\leq \|f\|_{p,k},\quad\text{where} \quad\frac{1}{q}=\frac{1}{p}-\frac{\alpha}{2k+1}$$
\end{thm}
\begin{proof}
 The fractional integral $ I_k^\alpha(f)$  can be expressed in the form of a convolution
$$I_k^\alpha(f)=f\star_k \mathcal{I}_k^\alpha$$
where 
$$\mathcal{I}_k^\alpha=\sum_{n\in\mathcal{Z},n\neq  0}|n|^{-\alpha}\rho_n \mathcal{E}_n^k, \quad \rho_n 
=(\| \mathcal{E}_n^k,\|_{2,k})^{-2}$$
It therefore suffices to show that
  $$\|\mathcal{I}_k^\alpha\|_{q,k}^*<\infty, \quad \text{for}\quad1-\frac{1}{q}=\frac{\alpha}{2k+1}.$$
For each $n\neq 0$ we invoke the well-known integral representation 
$$|n|^{-\alpha }=\frac{1}{\Gamma(\alpha)}\int_0^1r^{|n|}\;\ln(\dfrac{1}{r})^{\alpha-1}\frac{dr}{r} $$
to write 
$$\mathcal{I}_k^\alpha(x)=\frac{1}{\Gamma(\alpha)}\int_0^1\sum_{n\neq  0}n^{-\alpha}r^{|n|}\rho_n \mathcal{E}_n^k(ix)\ln(\dfrac{1}{r})^{\alpha-1}\frac{dr}{r} .$$
However, evaluating the Poisson kernel at $y=0$ yields
$$\sum_{n\neq  0}n^{-\alpha}r^{|n|}\rho_n \mathcal{E}_n^k(ix)=\sum_{n\neq 0}\gamma_n^k r^{|n|}E_n^k(ix)E_n^k(0)=\dfrac{\Gamma(k+1)}{2\sqrt{\pi}\Gamma(k+\frac{1}{2})}\left(\dfrac{1-r^2}{(1-2r\cos x+r^2)^{k+1} }-1\right).$$
Thus we obtain the integral representation
\begin{equation}\label{00}
\mathcal{I}_k^\alpha(x)=\dfrac{\Gamma(k+1)}{2\sqrt{\pi}\Gamma(k+\frac{1}{2})\Gamma(\alpha)} \int_0^1 \left[\dfrac{1-r^2}{1-2r\cos x+r^2 }-1\right]\ln(\dfrac{1}{r})^{\alpha-1}\frac{dr}{r} .
\end{equation}
Let $E\subset[-\pi,\pi]$,
$$E=E\cap[0,\pi]\cup E\cap [-\pi,0)=E_1\cup E_2.$$
Because $\mathcal{I}_k^\alpha$ and the measure $dm_k(x)$ are even, we have  
\begin{eqnarray*}
\int_{-\pi}^{\pi}|\mathcal{I}_k^\alpha(x)|\chi_{E}(x)dm_k(x)&=&\int_{0}^{\pi}|\mathcal{I}_k^\alpha(x)|\chi_{E_1}(x)dm_k(x)+\int_{-\pi}^{0}|\mathcal{I}_k^\alpha(x)|\chi_{E_2}(x)dm_k(x)
\\&=&\int_{0}^{\pi}|\mathcal{I}_k^\alpha(x)|\chi_{E_1}(x)dm_k(x)+\int_{0}^{\pi}|\mathcal{I}_k^\alpha(x)|\chi_{(-E_2)}(x)dm_k(x)
\end{eqnarray*}
and for all $q>1$
\begin{eqnarray}\label{000}\nonumber
\frac{\int_{\pi}^{\pi}|\mathcal{I}_k^\alpha|\chi_{E}dm_k}{\|\chi_{E}\|_{q',k}}&\leq & \frac{\int_{0}^{\pi}|\mathcal{I}_k^\alpha|\chi_{E}dm_k}{\|\chi_{E_1}\|_{q',k}}+\frac{\int_{0}^{\pi}|\mathcal{I}_k^\alpha|\chi_{(-E_2)}dm_k}{\|\chi_{(-E_2)}\|_{q',k}}
\\&\leq & 2\sup_{E\subset[0,\pi]}\frac{\int_{0}^{\pi}|\mathcal{I}_k^\alpha|\chi_{E}dm_k}{\|\chi_{E}\|_{q',k}} 
\end{eqnarray}
 In view of  (\ref{00}) and the proof Theorem 12 of  \cite{MS},
$$\sup_{E\subset[0,\pi]}\frac{\int_{0}^{\pi}|\mathcal{I}_k^\alpha|\chi_{E}dm_k}{\|\chi_{E}\|_{q',k}} <\infty$$ which implies
$$\|\mathcal{I}_k^\alpha\|_{q,k}^*<\infty.$$
 The proof is thus complete.
 \end{proof}
\section{A generalized Hilbert transform}
We introduce the first-order differential-difference operator
$$ \mathcal{T}_k(f)(x)= f'(x)+2ki\frac{f(x)-f(-x)}{1-e^{-2ix}}-kif(x)$$
 and define the corresponding Laplacian-type operator by
$$\Delta_k=\mathcal{T}_k^2.$$
When restricted to even functions
$\Delta_k$ reduces to the real second-order differential operator  
$$\Delta_k f(x)= f''(x)+2k\cot x f'(x)-k^2 f(x).$$
Moreover, the functions $E_n^k$ are eigenfunctions of $\mathcal{T}_k$, in the sense that
$$\mathcal{T}_k(E_n^k(ix))=in_k E_n^k(ix).$$
A direct computation shows that the adjoint of $\mathcal{T}_k$
  with respect to the measure $dm_k(x)$ satisfies
$$\mathcal{T}_k^*=-\mathcal{T}_k$$
Consequently, $-\Delta_k$ is  positive self-adjoint operator on   $L^2([-\pi,\pi], dm_k(x))$.
\par In analogy with the classical setting, we define the Hilbert transform associated with  $\Delta_k$ by
 $$H_k=\mathcal{T}_k(-\Delta_k)^{-\frac{1}{2}}.$$
For $f\in L^2([-pi,\pi], dm_k(x))$,
$$f(x)=\sum_{n\in \mathbb{Z}}a_n E_n^k(ix),$$
  we then have 
$$H_k(f)(x)=i\sum_{n\in \mathbb{Z}} \frac{n_k}{|n_k|} a_n E_n^k(ix)
=i\sum_{n=1}^\infty a_n E_n^k(ix)-i\sum_{n=-\infty}^0  a_n E_n^k(ix).$$
It follows immediately that the Hilbert  transform $H_k$  
is skew-adjoint, $ H_k^*=-H_k$ and is a bounded operator  on $L^2([-\pi,\pi], dm_k(x))$, that is 
$$\|H_k(f)\|_{2}\leq  \|f\|_2.$$
\begin{prop}\label{p22}
 The Hilbert transform $H_k$ is a singular integral operator. More precisely, for any 
  $f\in L^2([-\pi,\pi], dm_k(x))$, and  any $x$ such that $\pm x\notin supp(f)$, one has
  $$H_k(f)(x)=\int_{-\pi}^{\pi}\mathcal{H}_k(x,y)f(y)dm_k(y)$$
where  the kernel $\mathcal{H}_k$ is explicitly given, for $x\neq \pm y$
  by
  \begin{equation}\label{hk}
  \mathcal{H}_k(x,y)=i\frac{k2^{-2k}}{2\pi}(1-e^{i(x-y)} )\int_{-1}^1\dfrac{(1+u)(1-u^2 )^{k-1} du}{(1-(\cos x\cos y+u\sin x\sin y) )^{k+1}}.
  \end{equation}
In addition, $\mathcal{H}_k$ satisfies the standard singular estimate
\begin{equation}\label{hkk}
|\mathcal{H}_k(x,y)|\leq  \frac{C_k}{||x|-|y||^{2k+1}}
\end{equation}
where $C_k>0$ is a constant depending only on the multiplicity parameter $k$.
\end{prop}
To proceed, we begin with a simple geometric identity which will be used in the kernel estimates. 
\begin{lem}
For all $x,y\in[-\pi,\pi]$ we have
$$
\min\!\left(\sin\frac{|x-y|}{2},\;\sin\frac{|x+y|}{2}\right)
= \sin\frac{\,\big||x|-|y|\big|}{2}.$$
\end{lem}

\begin{proof}
Let $a=|x|$ and $b=|y|$. Without loss of generality we may assume $a\ge b$.
There are two possible sign configurations:
\begin{itemize}
    \item If $x$ and $y$ have the same sign, then $|x-y|=a-b$ and $|x+y|=a+b$.
    \item If $x$ and $y$ have opposite signs, then $|x+y|=a-b$ and $|x-y|=a+b$.
\end{itemize}

Thus, in both cases,
$$
\min\!\left(\sin\frac{|x-y|}{2},\;\sin\frac{|x+y|}{2}\right)
= \min\!\left(\sin\frac{a-b}{2},\;\sin\frac{a+b}{2}\right).
$$
We now show that
$$
\sin\frac{a+b}{2}\ge\sin\frac{a-b}{2}.
$$
Let
$$
C=\frac{a+b}{2},\qquad D=\frac{a-b}{2}.
$$
Then by the sine difference identity,
$$
\sin C - \sin D
= 2\cos\frac{C+D}{2}\,\sin\frac{C-D}{2}
= 2\cos\frac{a}{2}\,\sin\frac{b}{2}.
$$
Since $a,b\in[0,\pi]$, both factors on the right-hand side are nonnegative, so
$\sin C\ge\sin D$. Hence
$$
\min\!\left(\sin\frac{a-b}{2},\;\sin\frac{a+b}{2}\right)
= \sin\frac{a-b}{2}.
$$
Finally, since $a-b = \big||x|-|y|\big|$, we conclude
$$
\min\!\left(\sin\frac{|x-y|}{2},\;\sin\frac{|x+y|}{2}\right)
= \sin\frac{\,\big||x|-|y|\big|}{2},
$$
which completes the proof.
\end{proof}
\begin{proof}[ Proof of Proposition \ref{p22}.]
Let us write 
\begin{eqnarray*}
&&H_k(f)(x)\\&=&i\lim _{r\rightarrow 1^-}\sum_{n\in \mathbb{Z}}\frac{n_k}{|n_k|}r^{|n|+k} a_n E_n^k(ix)\\&=& i\lim _{r\rightarrow 1^-}\int_{-\pi}^\pi f(y)
\left(\sum_{n\in \mathbb{Z}}\frac{n_k}{|n_k|}r^{|n|+k} \gamma_n^{k} E_n^k(ix) E_n^k(-iy)\right)dm_k(y)
\\&=& i\lim _{r\rightarrow 1^-}\int_{-\pi}^\pi f(y)
\left(\sum_{n=1}^\infty \gamma_n^{k}r^{|n|+k} E_n^k(ix) E_n^k(-iy)-
\sum_{n=-\infty}^0 \gamma_n^{k} r^{|n|+k}E_n^k(ix) E_n^k(-iy)\right)dm_k(y).
\end{eqnarray*}
Let  
$$A=\sum_{n=1}^{\infty}r^{n+k} \gamma_n^{k} E_n^k(ix) E_n^k(-iy).$$
In view of (\ref{12}) we   obtain 
$$ \sum_{n=0}^{\infty}r^{n+k} \gamma_{-n}^{k} E_{-n}^k(ix) E_{-n}^k(-iy)
=\frac{e^{i(x-y)}}{r}\bar{A}.$$
As the Poisson kernel,
$$P_k(r,x,y)=A+\frac{e^{i(x-y)}}{r}\bar{A}=\bar{A}+\frac{e^{-i(x-y)}}{r}A,$$
$A$ and $\bar{A}$ can be written explicitly  as
$$A=\frac{(r^2-re^{i(x-y)})P_k(r,x,y)}{r^2-1}\quad\text{and} \quad \bar{A}=\frac{(r^2-re^{-i(x-y)})P_k(r,x,y)}{r^2-1}. $$
From which  it follows that 
\begin{eqnarray*}
 && i\sum_{n\in \mathbb{Z}} \frac{n_k}{|n_k|}r^{|n|+k} \gamma_n^{k} 
 E_n^k(ix) E_n^k(-iy)\\&=&i(A-\frac{e^{i(x-y)}}{r}\bar{A})\\&=&i\frac{1+r^2-2re^{i(x-y)} }{1-r^2 }P_k(r,x,y)\\
    & =&ir^k\frac{k}{2\pi}(1+r^2-2re^{i(x-y)} )\int_{-1}^1\dfrac{(1+u)(1-u^2 )^{k-1} du}{(1-2r(\cos x\cos y+u\sin x\sin y)+r^2 )^{k+1}}.
     \end{eqnarray*}
It may be noted that for $r\in [0,1]$,  $u\in [-1,1]$ and $x\neq \pm y$,
$$1-2r(\cos x\cos y-u\sin x\sin y)+r^2\geq 
4r \min \left(\sin^2\frac{|x-y|}{2}, \sin^2\frac{|x+y|}{2}\right)=4r \sin^2\frac{||x|-|y||}{2}.$$
Consequently one obtains the kernel bound
$$\left|\sum_{n\in \mathbb{Z}}\frac{n_k}{|n_k|}r^{|n|+k} \gamma_n^{k} E_n^k(ix) E_n^k(-iy)\right|\leq 
C_k\frac{1}{\left(\sin \frac{||x|-|y||}{2}\right)^{2k+2}}\leq C_k\frac{1}{||x|-|y||^{2k+2}},$$
(Using the elementary lower bound $\sin t\geq \dfrac{2}{\pi}t$ , for $t\in[0,\dfrac{\pi}{2}]$).
Therefore, when $\pm x\notin supp(f)$, the dominated convergence theorem applies, allowing us to write
$$H_k(f)(x)=\int_{-\pi}^{\pi }\mathcal{H}_k(x,y)f(y)dm_k(y)$$
with $\mathcal{H}_k$ is given by (\ref{hk}). To  prove  (\ref{hkk}), we argue as follows. First note that, for every  $u\in [-1,1]$ and $x\neq \pm y$,
$$\dfrac{(1+u)}
{1-\cos x\cos y-u\sin x\sin y }\leq \frac{1}{\sin^2 \frac{x-y}{2}}.
$$
Combining this estimate with (\ref{hk}) yields
$$|\mathcal{H}_k(x,y)|\leq C_k\frac{|\sin\frac{x-y}{2}|}{\sin^2 \frac{x-y}{2}\sin^{2k} \frac{||x|-|y||}{2}}\leq C_k\frac{1}{\sin^{2k+1} \frac{||x|-|y||}{2}}
\leq C_k \frac{1}{||x|-|y||^{2k+1}}.$$
This concludes the proof.
\end{proof}
\begin{thm}\label{lp}
The Riesz transform $H_k $extends to a bounded operator on $L^p([-\pi,\pi],dm_k)$
 for $1<p<\infty$ and is of weak type
$(1,1)$.
\end{thm}
\begin{proof}
The proof proceeds by embedding $H_k$ into the Calderón–Zygmund framework.
 The measure $m_k$ is a symmetric doubling on $[-\pi,\pi]$. More precisely, we shall make use of a Hormander-type regularity condition, in the form established and applied in \cite{B1, B2}. Accordingly, the key ingredient of the proof is to establish that the kernel  $\mathcal{H}_k(x,y)$ satisfies away from the diagonals $x\neq \pm y$, the estimate
 \begin{equation}\label{Hor}
 \int_{||x|-|y||>2|y-y'|}|\mathcal{H}_k(x,y)-\mathcal{H}_k(x,y')|\leq C_k,
 \end{equation}
for some constant $C_k>0$  independent of $y$ and $y'$.
 Throughout the argument, we also  denote by $C_k$ for a positive constant whose value may vary from line to line.
\par Let us begin by writing
$$\mathcal{H}_k(x,y)-\mathcal{H}_k(x,y')=(y-y')
\int_0^1\frac{\partial \mathcal{H}_k}{\partial y}(x,y_t)dt$$
where  $y_t=y'+t(y-y')$. The derivative of the kernel is given by
\begin{eqnarray*}
&&\frac{\partial \mathcal{H}_k}{\partial y}(x,y)\\&&=-\frac{k2^{-2k}}{2\pi}e^{i(x-y)}\int_{-1}^1\dfrac{(1+u)(1-u^2 )^{k-1} du}{(1-\cos x\cos y-u\sin x\sin y )^{k+1}}
\\&&-\frac{k2^{-2k}}{2\pi}(1-e^{i(x-y)}\int_{-1}^1\dfrac{(1+u)(1-u^2 )^{k-1} 
(\cos x\sin y-u\sin x\cos y) du}{(1-\cos x\cos y-u\sin x\sin y )^{k+2}}
\\&&=I_1(x,y)+I_2(x,y).
\end{eqnarray*}
Notice that under the assumption 
 $||x|-|y||\geq  2|y-y'|$  we have that  
 $$||x|-|y_t||\geq  |y-y'|$$
and hence
  $$1-\cos x\cos y_t-u\sin x\sin y_t\geq 
 \sin^2\frac{||x|-|y_t||}{2}\geq \sin^2\frac{|y-y'|}{2}.$$
 Consequently, we obtain the estimate
$$|I_1(x,y_t)|\leq C_k
 \int_{-1}^1\dfrac{(1-u^2 )^{k-1} du}{(\sin^2|y-y'|+1-\cos x\cos y_t-u\sin x\sin y_t )^{k+1}}. $$
 and therefore, 
 \begin{eqnarray}\label{11}\nonumber
 &&|y-y'|\int_{||x|-|y||>2|y-y'|}\int_0^1|I_1(x,y_t)|dtdm_k(x)\\\nonumber
 &\leq& C_k|y-y'|\int_0^1 \int_{-\pi}^\pi
 \int_{-1}^1\dfrac{(1-u^2 )^{k-1} du}{(\sin^2|y-y'|+1-(\cos x\cos y_t+u\sin x\sin y_t) )^{k+1}}du dm_k(x)dt\\ \nonumber
 &\leq& C_k |y-y'|\int_{0}^\pi
 \dfrac{sin^{2k} x}{(\sin^2|y-y'|+1-\cos x )^{k+1}} dx\\
 &\leq &C_k |y-y'|\int_{0}^{^{+\infty}}
 \dfrac{t^{2k}}{(|y-y'|^2+ t^2 )^{k+1}} dt\leq C_k  
 \end{eqnarray}
  where  we have used the standard bounds, for $t\in [0,\pi]$
$\sin t\leq t $ ,$1-\cos t=2\sin^2\frac{t}{2}$ .Also for $t\in [0,\pi/2]$,  $\sin t\geq 2t/\pi$. This completes the estimate for the integral associated with 
$I_1$, establishing the desired bound for this part of the Hörmander condition.
It is appropriate here to clarify the second inequality of (\ref{11}). In fact, let
$$g(x)=\frac{1}{\sin^2|y-y'|+1-x}, \qquad |x|\leq 1.$$
Making the  substitution
$z=\arccos(\cos x\cos y+u\sin x\sin y)$
we obtain
$$\int_{-1}^1 g(\cos x\cos y+u\sin x\sin y )(1-u^2)^{k-1}du=\frac{2\Gamma(k)\sqrt{\pi}}{\Gamma(k+\dfrac{1}{2})}\int_{0}^\pi g(\cos z)W_k(x,y,z)dm_k(z).$$
 Integrating both sides of with respect to $dm_k(x)$, and using the normalization property
 $$\int_{-\pi}^\pi W_k(x,y,z)dm_k(x)=1$$
  it follows that
 $$\int_{-\pi}^\pi\left\{\int_{-1}^1 g(\cos x\cos y+u\sin x\sin y )(1-u^2)^{k-1}\right\}dudm_k(x)=\frac{2\Gamma(k)\sqrt{\pi}}{\Gamma(k+\dfrac{1}{2})}\int_{0}^\pi g(\cos z)dm_k(z).$$
 \par  Turning now  to the term $I_2(x,y)$, for which we distinguish two separate cases for $x,y\in [-\pi,\pi]$.\\
\textbf{Case 1.} $x$ and $y$ have the same sign. Introduce the auxiliary function, 
$$ A(u)=(1-e^{i(x-y)})\dfrac{\cos x\sin y-u\sin x\cos y}{1-\cos x\cos y-u\sin x\sin y }$$
A direct inspection shows that
$$|A(u)|\leq \max (|A(-1)|,|A(1)|).$$
For $u=1$ we compute
$$|A(1)|= \left|\frac{\sin\frac{x-y}{2}\sin(x-y)}{\sin^2\dfrac{(x-y)}{2}}\right|=
2\cos\frac{x-y}{2}\leq 2.$$
While for $u=-1$ we obtain
$$|A(-1)|= \left|\frac{\sin\frac{x-y}{2}\sin(x+y)}{\sin^2\dfrac{(x+y)}{2}}\right| =2\left|\frac{\sin\frac{x-y}{2}\cos\frac{x+y}{2}}{\sin\dfrac{(x+y)}{2}}\right|\leq  2\cos\frac{x+y}{2}\leq 2.$$
Hence 
$$|A(u)|\leq 2, \qquad \text{for all}\; u\in [-1,1].$$
\textbf{Case 2.} $x$ and $y$ have opposite signs. We have for $u\in[-1,1]$,
$$\left|\dfrac{(1+u)}
{1-\cos x\cos y-u\sin x\sin y }\right|\leq \frac{1}{\sin^2 \frac{x-y}{2}}.$$
Since
$$|\cos x\sin y-u\sin x\cos y|\leq \max (|\sin(x-y)|,|\sin(x+y)|)$$
and 
$$\left|\sin \frac{x+y}{2}\right|\leq \left|\sin \frac{x-y}{2}\right|$$
we  obtain that 
$$ \left|(1-e^{i(x-y)})(1+u)\dfrac{\cos x\sin y-u\sin x\cos y}{1-\cos x\cos y-u\sin x\sin y }\right|\leq 4.$$
Combining the two cases, we deduce
$$|I_2(x,y)\leq  C\int_{-1}^1\dfrac{(1-u^2 )^{k-1} 
 }{(1-\cos x\cos y-u\sin x\sin y )^{k+1}} du.$$
Arguing in a manner analogous to that used for $I_1$, we obtain
$$|y-y'|\int_{||x|-|y||>2|y-y'|}\int_0^1|I_2(x,y_t)|dtdm_k(x)\leq C_k.$$
This completes the proof of (\ref{Hor}) and hence Theorem \ref{lp} follows.
\end{proof}
In conclusion, we point out that when restricted to even functions, Theorem \ref{lp} can be seen as an extension of the corresponding result for the Riesz transform associated with ultraspherical polynomials, as established in \cite{DB}.

 \end{document}